\documentclass[11pt, regno]{amsart}
\usepackage{amsmath, mathrsfs, amsthm, amscd, amsfonts, amssymb, graphicx, graphics,color}
\usepackage[breaklinks,colorlinks=true,allcolors=cyan,citecolor=red,pagebackref=false,hyperindex=false]{hyperref}
\usepackage[margin=2.9cm]{geometry}
\usepackage[mathscr]{eucal}
\usepackage{t1enc}

\newtheorem{thm}{Theorem}[section]
\newtheorem{cor}[thm]{Corollary}
\newtheorem{lemma}[thm]{Lemma}
\newtheorem{prop}[thm]{Proposition}
\theoremstyle{definition}
\newtheorem{defn}[thm]{Definition}
\theoremstyle{notation}
\newtheorem{notation}[thm]{\bf Notation}
\theoremstyle{remark}
\newtheorem{rem}[thm]{Remark}
\numberwithin{equation}{section}
\newtheorem{example}[thm]{Example}

\newcommand{\h}{\mathcal{H}}
\newcommand{\K}{\mathcal{K}}

\newcommand{\vv}[2]{\begin{pmatrix}%
		#1	\\%
		#2	%
\end{pmatrix} }
\DeclareMathOperator{\spn}{span}

\DeclareMathOperator{\ran}{\mathcal{R}}

\begin{document}
	\title[Fibonacci representations of sequences in Hilbert spaces]{Fibonacci representations of sequences in Hilbert spaces}
\author[J. Sedghi Moghaddam, A. Najati and Y. Khedmati ]{J. Sedghi Moghaddam, A. Najati and Y. khedmati}
\address{
 Department of Mathematics
\newline
\indent Faculty of  Sciences
\newline
\indent  University of  Mohaghegh Ardabili
\newline \indent  Ardabil, Iran
  }
\email{\quad j.smoghaddam@uma.ac.ir,\quad a.nejati@yahoo.com, \quad khedmati.y@uma.ac.ir }

\begin{abstract}
Dynamical sampling deals with frames of the form $\{T^n\varphi\}_{n=0}^\infty$, where
$T \in B(\h)$ belongs to certain classes of linear operators and $\varphi\in\h$.
The purpose of this paper is to investigate a new representation, namely, Fibonacci representation of sequences $\{f_n\}_{n=1}^\infty$ in a Hilbert space $\h$; having the form $f_{n+2}=T(f_n+f_{n+1})$ for all $n\geqslant 1$ and a linear operator $T :\spn\{f_n\}_{n=1}^\infty\to\spn\{f_n\}_{n=1}^\infty$. We apply this kind of representations for complete sequences and frames. Finally, we present some properties of Fibonacci representation operators.
\end{abstract}
\maketitle
\section{{\textbf Introduction}}
We recall some definitions and standard results from frame theory.
\begin{defn}
	\begin{enumerate}Consider a sequence $F=\{f_i\}_{i=1}^\infty$ in $\h$.
		\item[$(i)$] $F$ is called a \textit{frame} for $\h$, if there exist two constants $A_F, B_F > 0$ such that
		\begin{eqnarray*}\label{abc}
		A_F \|f\|^2\leqslant \sum_{i=1}^\infty |\langle f,f_i\rangle |^2 \leqslant  B_F \|f\|^2,\quad f\in \h.
		\end{eqnarray*}
		\item [$(ii)$] $F$ is called a \textit{Bessel sequence} with Bessel bound $B_F$ if at least the upper frame condition holds.
		\item [$(iii)$] $F$ is called \textit{complete} in $\h$ if $\overline{\spn}\{f_i\}_{i=1}^\infty=\h,$ i.e.,  $\spn\{f_i\}_{i=1}^\infty$ is dense in $\h$.
\item [$(iv)$] $F$ is called
\textit{linearly independent} if $\sum_{k=1}^m c_kf_k=0$  for some scalar
coefficients $\{c_k\}_{k=1}^m$, then $c_k =0$ for all $k = 1,\cdots, m$. We say $F$ is \textit{linearly dependent} if $F$ is not linearly independent.
	\end{enumerate}
\end{defn}
\begin{thm}\cite[Theorem 5.5.1]{Ole_Book}\label{onto}
	A sequence $F=\{f_i\}_{i=1}^\infty\subseteq \h $ is a frame for $\h$ if and only if
	\begin{equation*}
	T_F:\ell^2 \rightarrow\h,\quad T_F(\{c_i\}_{i=1}^\infty )=\sum_{i=1}^\infty  c_if_i,
	\end{equation*}
	is a well-defined mapping from $\ell^2$ onto $\h$. Moreover, the adjoint of $T_F$ is given by
\begin{eqnarray*}
	T^*_F:\h\rightarrow \ell^2 ,\quad T_F^* f=\{\langle f,f_i \rangle\}_{i=1}^\infty .
\end{eqnarray*}
\end{thm}
It is well known that if $F=\{f_i\}_{i=1}^\infty\subseteq \h $ is a Bessel sequence for $\h$, then $T_F$ is well-defined and bounded.
\par
In \cite{Ald_Acha}, Aldroubi et al. introduced  dynamic sampling which it deals with frame properties of sequences of the form $\{T^n\varphi\}_{n=0}^\infty$, where
$T \in B(\h)$ belongs to certain classes of linear operators (such as diagonalizable normal operators) and $\varphi\in\h$.
 Various characterizations of frames having the form
$ \{f_{k}\}_{k\in I}=\{T^k\varphi\}_{k\in I},$
 where $T$ is a linear (not necessarily bounded) operator can be found in
\cite{O.M_Int, O.M_Acha,  O.M.E_Jmaa, RNO}.
\begin{prop}\cite[Proposition 2.3]{O.M_Acha}\label{mnref}
	Consider a frame sequence $F=\{f_i\}_{i=1}^\infty$ in a Hilbert space $\h$  which spans an infinite-dimensional subspace.
	The following are equivalent:
	\begin{enumerate}
		\item[$(i)$] $F$ is linearly independent.
		\item[$(ii)$] There exists a linear operator $T:\spn\{f_i\}_{i=1}^\infty\rightarrow \h$ such that $Tf_i:=f_{i+1}$.
	\end{enumerate}
\end{prop}
\begin{thm}\cite[Theorem 2.1]{O.M.E_Jmaa}\label{exbnd}
	Consider a frame $F=\{f_i\}_{i=1}^\infty$ in $\h$. Then the following are equivalent:
	\begin{enumerate}
		\item[$(i)$] $F=\{T^{i-1} f_1\}_{i=1}^\infty$ for some $T\in B(\h)$.
		\item[$(ii)$] For some dual frame $G=\{g_i\}_{i=1}^\infty$ (and hence all)
		\begin{align*}
		f_{j+1}=\sum_{i=1}^\infty\langle f_j,g_i\rangle f_{i+1},\forall j\in\mathbb{N}.
		\end{align*}
		\item[$(iii)$] The $\ker T_F$ is invariant under the right-shift operator $\mathcal{T}:\ell^2\to\ell^2$ defined by \[\mathcal{T}(c_1, c_2,\cdots)=(0, c_1, c_2,\cdots).\]
	\end{enumerate}
\end{thm}
\section{Special sequences }
It is well known, cf. \cite[Example 5.4.6]{Ole_Book} that if $\{e_n\}_{n=1}^\infty$ is an orthonormal basis $\{e_n\}_{n=1}^\infty$ for $\h$, then $\{e_n+e_{n+1}\}_{n=1}^\infty$ is complete and a Bessel sequence but  not a frame. This motivates us to  investigate some results concerning the sequences $F=\{f_n\}_{n=1}^\infty$, $M=\{f_n+f_{n+1}\}_{n=1}^\infty$ and $N=\{f_n-f_{n-1}\}_{n=1}^\infty$ in a Hilbert space $\h$.
\begin{prop}\label{lem1}
	Let $\alpha$ and $\beta$ be nonzero scalers and $F=\{f_n\}_{n=1}^\infty \subseteq \h$. Then
	\begin{enumerate}
		\item [$(i)$] $F$ is a Bessel sequence for $\h$, if and only if  $M=\{\alpha f_n +\beta f_{n+1}\}_{n=1}^\infty$ and $N=\{\alpha f_n -\beta f_{n+1}\}_{n=1}^\infty$ are Bessel sequences for $\h$.
		\item [$(ii)$] Suppose that $F$ is a Bessel sequence for $\h$. Then $F$ is complete, if and only if $M=\{\alpha f_n +\beta f_{n+1}\}_{n=1}^\infty$ is complete, whenever $|\alpha|\geqslant  |\beta|$.
	\end{enumerate}
\end{prop}
\begin{proof}
	\begin{enumerate}
		\item [$(i)$]  Assume that $\{f_n\}_{n=1}^\infty$ is a Bessel sequence with  Bessel bound $B_F$ and $\mu= \max\{|\alpha|^2,|\beta|^2\}$. Then for $f\in\h$, we have
		\begin{align*}
		\sum_{n=1}^\infty |\langle f, \alpha f_n+\beta f_{n+1}\rangle|^2&+\sum_{n=1}^\infty |\langle f, \alpha f_n-\beta f_{n+1}\rangle|^2\\
		&= 2|\alpha|^2 \sum_{n=1}^\infty |\langle f, f_n\rangle|^2 +2|\beta|^2\sum_{n=1}^\infty |\langle f, f_{n+1}\rangle|^2\\
		& \leqslant  4 \mu \sum_{n=1}^\infty |\langle f, f_n\rangle|^2 \leqslant  4\mu B_F \|f\|^2.
		\end{align*}
		Then $M$ and $N$ are Bessel sequences. For the opposite implication, let $B_M$ and $B_N$ be Bessel bounds for sequences $M$ and $N$, respectively. Then
		\begin{align*}
		2|\alpha|^2\sum_{n=1}^\infty |\langle f, f_n\rangle|^2&\leqslant  2  \sum_{n=1}^\infty |\langle f, \alpha f_n\rangle|^2+2\sum_{n=1}^\infty |\langle f, \beta f_{n+1}\rangle|^2\\
		&=\sum_{n=1}^\infty |\langle f, \alpha f_n +\beta f_{n+1} \rangle|^2 +\sum_{n=1}^\infty |\langle f, \alpha f_n -\beta f_{n+1}\rangle|^2\\
		&\leqslant  (B_M+B_N)\|f\|^2, \quad f\in \h.
		\end{align*}
		\item [$(ii)$] Suppose that $F$ is complete and $f\in \h$ such that $\langle f, \alpha f_n +\beta f_{n+1}\rangle =0$ for all $n\in \mathbb{N}$. Then $\overline{\alpha}\langle f, f_n\rangle=-\overline{\beta} \langle f, f_{n+1}\rangle$ for all $n\in \mathbb{N}$. Since $|\alpha|\geqslant |\beta|$ and
\[|\langle f,f_1 \rangle|^2 \sum_{n=0}^\infty \big|\dfrac{\alpha}{\beta}\big|^n =\sum_{n=1}^\infty |\langle f, f_n\rangle|^2 \leqslant  B_F \|f\|^2,\]
		we get $\langle f, f_1\rangle=0$ and consequently  $\langle f, f_n\rangle=0$ for $n\in\mathbb{N}$. Hence $f=0$ and this shows that $\{\alpha f_n +\beta f_{n+1}\}_{n=1}^\infty$ is complete. In order to show the other implication, assume that $M$ is complete and $f\in \h$ such that $\langle f,  f_n \rangle =0$ for all $n\in \mathbb{N}$. Since
\[\langle f, \alpha f_n +\beta f_{n+1}\rangle =\overline{\alpha}\langle f, f_n\rangle+\overline{\beta} \langle f, f_{n+1}\rangle=0,\quad n\in\mathbb{N},\]
		we conclude that $f=0$ and therefore $F$ is complete.
	\end{enumerate}
\end{proof}
\begin{prop}
	Let $F=\{f_n\}_{n=1}^\infty $, $M=\{\alpha f_n +\beta f_{n+1}\}_{n=1}^\infty$ and $N=\{\alpha f_n -\beta f_{n+1}\}_{n=1}^\infty $ be sequences in a Hilbert space $\h$ and $\alpha\neq0$. Then $F$ is a frame for $\h$, if and only if $M\cup N$ is a frame for $\h$.
\end{prop}
\begin{proof} Let $\mu= \max\{|\alpha|^2,|\beta|^2\}$. Then the result  follows from
	\begin{align*}
 |\alpha|^2\sum_{n=1}^\infty |\langle f, f_n\rangle|^2&\leqslant  2  \sum_{n=1}^\infty |\langle f, \alpha f_n\rangle|^2+2\sum_{n=1}^\infty |\langle f, \beta f_{n+1}\rangle|^2\\
	&=\sum_{n=1}^\infty |\langle f, \alpha f_n +\beta f_{n+1} \rangle|^2 +\sum_{n=1}^\infty |\langle f, \alpha f_n -\beta f_{n+1}\rangle|^2\\
	&= 2|\alpha|^2 \sum_{n=1}^\infty |\langle f, f_n\rangle|^2 +2|\beta|^2\sum_{n=1}^\infty |\langle f, f_{n+1}\rangle|^2\\
	& \leqslant  4 \mu \sum_{n=1}^\infty |\langle f, f_n\rangle|^2, \quad f\in\h.
	\end{align*}
\end{proof}
\begin{thm}
	Let $M=\{f_n +f_{n+1}\}_{n=1}^\infty$ and $N=\{f_n-f_{n+1}\}_{n=1}^\infty$ be frames for $\h$. Then $F=\{f_n\}_{n=1}^\infty$ is a frame for $\h$ and
	\begin{equation}\label{S_f=}
	4S_Ff=S_Mf+S_Nf+2\langle f, f_1\rangle f_1, \quad f\in \h,
	\end{equation}
where $S_F, S_M$ and $S_N$ are frame operators for $F, M$ and $N$, respectively.
\end{thm}
\begin{proof} By Proposition \ref{lem1}, $F$ is a Bessel sequence for $\h$. Let $A_M$ and $A_N$ be lower frame bounds for $M$ and $N$, respectively. Then we have
	\begin{align*}
	(A_M+A_N)\|f\|^2&\leqslant  \sum_{n=1}^\infty |\langle f, f_n+f_{n+1}\rangle|^2+\sum_{n=1}^\infty |\langle f, f_n-f_{n+1}\rangle|^2 \nonumber\\
	&= 2 \sum_{n=1}^\infty |\langle f, f_n\rangle|^2+ 2 \sum_{n=1}^\infty |\langle f, f_{n+1}\rangle|^2 \label{combes}\\
	&\leqslant  4 \sum_{n=1}^\infty |\langle f, f_n\rangle|^2,  \quad f\in \h \nonumber.
	\end{align*}
	Therefore, $F$ is  a frame for $\h$. Furthermore, since
\begin{align*}
	&\sum_{n=1}^\infty |\langle f, f_n+f_{n+1}\rangle|^2+\sum_{n=1}^\infty |\langle f, f_n-f_{n+1}\rangle|^2 \nonumber\\
	&= 2 \sum_{n=1}^\infty |\langle f, f_n\rangle|^2+ 2 \sum_{n=1}^\infty |\langle f, f_{n+1}\rangle|^2,  \quad f\in \h,
	\end{align*}
we obtain
	\begin{align*}
	\langle S_Mf ,f\rangle +\langle S_Nf ,f\rangle&=4\langle S_Ff ,f\rangle-2|\langle f ,f_1\rangle|^2\\
	&=4\langle S_Ff ,f\rangle -2\langle \langle f ,f_1\rangle f_1 ,f\rangle, \quad f\in \h.
	\end{align*}
	Then we obtain (\ref{S_f=}).
\end{proof}
\begin{thm}
	Let $M=\{f_n+f_{n+1}\}_{n=1}^\infty$, $N=\{f_n-f_{n+1}\}_{n=1}^\infty$ and $F=\{f_n\}_{n=1}^\infty$ be Bessel sequences for $\h$ and $\ker T_F$ be invariant under $\mathcal{T}_R$. Then
$\ker T_F=\ker T_M\cap\ker T_N$.
\end{thm}
\begin{proof}
	Let $\{c_n\}_{n=1}^\infty  \in \ker T_F$. Since $\ker T_F$ is invariant under $\mathcal{T}_R$, we get $\{0,c_1,c_2,...\} \in \ker T_F$. Therefore
$\{c_1,c_1+c_2,c_2+c_3,...\} , \{c_1,c_2-c_1,c_3-c_2,...\} \in \ker T_F$. Hence
	\begin{align*}
	\sum_{n=1}^\infty c_n (f_n+f_{n+1})&=\sum_{n=1}^\infty (c_n+c_{n+1})f_{n+1} +c_1 f_1 =0,\\
\sum_{n=1}^\infty c_n (f_n-f_{n+1})&=\sum_{n=1}^\infty (c_{n+1}-c_{n})f_{n+1} +c_1 f_1=0.
	\end{align*}
	Then we conclude  $\{c_n\}_{n=1}^\infty  \in \ker T_M\cap \ker T_N$.
	On the other hand, if $\{c_n\}_{n=1}^\infty  \in \ker T_M\cap \ker T_N$, then we have
	\begin{equation*}
	0=\sum_{n=1}^\infty c_n (f_n-f_{n+1})+\sum_{n=1}^\infty c_n (f_n+f_{n+1})=2\sum_{n=1}^\infty c_n f_n.
	\end{equation*}
	Therefore, $ \{c_n\}_{n=1}^\infty  \in \ker T_F$.
\end{proof}
\section{Fibonacci representation}
In this section we want to consider representation of a sequence $\{f_n\}_{n=1}^\infty\subseteq\h$ on the form $f_n=T(f_{n-1}+f_{n-2})$ for $n\geqslant 3$, where $T$ is a linear operator defined on  an appropriate
subspace of $\h$.
\begin{defn} We say that a sequence $F=\{f_n\}_{n=1}^\infty$ has  a \textit{Fibonacci representation} if there is a linear operator $T:\spn\{f_n\}_{n=1}^\infty\rightarrow \spn\{f_n\}_{n=1}^\infty$ such that $f_n=T(f_{n-1}+f_{n-2})$ for $n\geqslant 3$. In the affirmative case, we say that $F$ is represented by $T$, and $T$ is called a \textit{Fibonacci representation operator} with respect to $F$.
\end{defn}
Throughout this segment, $\h$ denotes a Hilbert space and $\{e_n\}_{n=1}^\infty$ is an orthonormal basis for $\h$.
\begin{example}
	It is clear that $F=\{f_n\}_{n=1}^\infty=\{e_1, e_1, e_2,...\}$ is a frame for $\h$. We define the linear operator $T:\spn\{e_n\}_{n=1}^\infty \rightarrow \spn\{e_n\}_{n=1}^\infty$ by
\[ Te_1=\dfrac{e_2}{2},\quad Te_n=\sum_{i=0}^{n-2}(-1)^i e_{n-i+1}+(-1)^{n+1}\dfrac{e_2}{2}, \quad n\geqslant  2. \]
Then $F$ is represented by $T$. Indeed, $Te_1+Te_2=e_3$  and 
for $n\geqslant  4$ we have
	\begin{align*}
	Te_{n-2}+Te_{n-1}&=\sum_{i=0}^{n-4}(-1)^i e_{n-i-1}+(-1)^{n-1}\dfrac{e_2}{2}+\sum_{i=0}^{n-3}(-1)^i e_{n-i}+(-1)^{n}\dfrac{e_2}{2}\\
	&=\sum_{i=0}^{n-4}(-1)^i e_{n-i-1}+e_{n}+\sum_{i=1}^{n-3}(-1)^i e_{n-i}\\
	&=\sum_{i=0}^{n-4}(-1)^i e_{n-i-1}+\sum_{i=0}^{n-4}(-1)^{i+1} e_{n-i-1}+e_{n}\\
	&=e_n.
	\end{align*}
	Note that $F$ is not linearly independent, and so by  \cite[Proposition 2.3]{O.M_Acha}, there does not exist a linear operator  $S:\spn\{e_n\}_{n=1}^\infty \rightarrow \spn\{e_n\}_{n=1}^\infty $ such that $Se_1=e_1$ and $Se_{n-1}=e_n,\ n\geqslant  2 $.
\end{example}
\begin{example}\label{exam}
		 The frame $F=\{f_n\}_{n=1}^\infty=\{e_1,e_2,e_3,e_1, e_4, e_5, e_6 ,...\}$ is represented by $T$, where $T:\spn\{f_n\}_{n=1}^\infty \rightarrow \spn\{f_n\}_{n=1}^\infty$ is defined by
		\begin{align*}
		Te_1&=\frac{1}{2}(e_4+e_3-e_1),\quad Te_2=\frac{1}{2}(-e_4+e_3+e_1),\quad Te_3=\frac{1}{2}(e_4-e_3+e_1),\\
		Te_4&=e_5-Te_1,\quad Te_n=e_{n+1}-Te_{n-1},\quad n\geqslant  5.
		\end{align*}
\end{example}
\begin{prop}
	A sequence $F=\{f_n\}_{n=1}^\infty $ is represented by $T$, if and only if $M=\{f_n+f_{n+1}\}_{n=1}^\infty$ and $N=\{f_n-f_{n+1}\}_{n=1}^\infty $ are represented by $T$.
\end{prop}
\begin{proof}
	First, let $F$ be represented by $T$. For every $n\in \mathbb{N}$, we have
	\begin{align*}
T\big((f_n+f_{n+1})+(f_{n+1}+f_{n+2})\big)&=T(f_n+f_{n+1})+T(f_{n+1}+f_{n+2})\\
	&=f_{n+2}+f_{n+3},\\
	T\big((f_n-f_{n+1})+(f_{n+1}-f_{n+2})\big)&=T(f_n+f_{n+1})-T(f_{n+1}+f_{n+2})\\
	&=f_{n+2}-f_{n+3},
	\end{align*}
	Then $M$ and $N$ are represented by $T$.
	Conversely, if $M$ and $N$ are represented by $T$, then
	\begin{align*}
	T(f_n&+f_{n+1})=\dfrac{1}{2}T(f_n+f_{n+1}+f_n-f_{n+1}+f_{n+1}+f_{n+2}+f_{n+1}-f_{n+2})\\
	&=\dfrac{1}{2}\Big(T(f_n+f_{n+1}+f_{n+1}+f_{n+2})+T(f_n-f_{n+1}+f_{n+1}-f_{n+2})\Big)\\
	&=\dfrac{1}{2}(f_{n+2}+f_{n+3}+f_{n+2}-f_{n+3})= f_{n+2}, \quad n\in \mathbb{N}.
	\end{align*}
Hence $F$ is represented by $T$.
\end{proof}
A frame may have more than one Fibonacci representation and a frame may not have  any.
\begin{example}\label{exam}
	 The frame $G=\{f_n\}_{n=1}^\infty=\{e_1,e_2,e_1,e_3,e_4,...\}$ does not have any Fibonacci representations. Indeed, if $G$ is represented by $T$, then
		\begin{align*}
		Te_1+Te_2=e_1,\quad Te_2+Te_1=e_3,
		\end{align*}
		which is a contradiction. We note that $\{f_n+f_{n+1}\}_{n=1}^\infty$ is not linearly independent.
\end{example}
\begin{example}\label{exam}
	 Consider the frame $E=\{e_n\}_{n=1}^\infty\subseteq\h$ and let $T,S:\spn\{e_n\}_{n=1}^\infty\to\spn\{e_n\}_{n=1}^\infty$ be linear operators defined by
\begin{align*}
	Te_1&=Te_2=\frac{1}{2}e_3,\quad Te_n=\frac{(-1)^n}{n}e_3-\sum_{i=4}^{n+1}(-1)^{n+i}e_i,\quad n\geqslant 4,\\
Se_1&=0,~ Se_2=e_3,~ Se_3=e_4-e_3,\quad Se_n=e_3-\sum_{i=4}^{n+1}(-1)^{n+i}e_i,\quad n\geqslant 4.
	\end{align*}
 Then it is easy to see that $E$ is represented by $T$ and $S$. We note that $\{e_n+e_{n+1}\}_{n=1}^\infty$ is  linearly independent.
\end{example}
In general if $\{f_n\}_{n=1}^\infty\subseteq\h$ is linearly independent with a Fibonacci representation $T$, then for each $g\in\spn\{f_n\}_{n=1}^\infty$ the linear operator $S:\spn\{f_n\}_{n=1}^\infty\to\spn\{f_n\}_{n=1}^\infty$ defined by
\[S\Big(\sum_{i=1}^k c_i f_i\Big)=\sum_{i=1}^k c_i Tf_i+\sum_{i=1}^k (-1)^ic_i g\]
is a Fibonacci representation for $\{f_n\}_{n=1}^\infty$.
\par
Now, we want to get a sufficient condition for a frame $F=\{f_n\}_{n=1}^\infty$ to have a Fibonacci representation.
We need the following lemma.
\begin{lemma}\label{lemf}
	Consider a sequence $\{f_n\}_{n=1}^\infty$ in $\h$. Then the following hold:
		\begin{enumerate}
\item[$(i)$] For $n\geqslant 2$, we have
		\[ f_n=\sum_{i=0}^{m-1}(-1)^i (f_{n-i-1}+f_{n-i})+(-1)^{m} f_{n-m},\quad 1\leqslant  m \leqslant  n-1.\]
		\item[$(ii)$] $\spn\{f_n\}_{n=1}^\infty=\spn\big\{\{f_1\} \cup\{f_n+f_{n+1}\}_{n=1}^\infty\big\}$.
		\item[$(iii)$] If $\{f_n\}_{n=1}^\infty$ is linearly independent, then $\{f_n+f_{n+1}\}_{n=1}^\infty$ is linearly independent.
		\item[$(iv)$] If $\{f_1\} \cup\{f_n+f_{n+1}\}_{n=1}^\infty$ is linearly independent, then $\{f_n\}_{n=1}^\infty$ is linearly independent.
	\end{enumerate}
\end{lemma}
\begin{proof}
$(i)$ Let $n\geqslant 2$ and $1\leqslant  m \leqslant  n-1$. Then we have
	\begin{align*}
	\sum_{i=0}^{m-1}(-1)^i (f_{n-i-1}+f_{n-i})
	&=\sum_{i=0}^{m-1}(-1)^i f_{n-i-1}+\sum_{i=0}^{m-1}(-1)^i f_{n-i}\\
	&=\sum_{i=1}^{m}(-1)^{i-1} f_{n-i}+\sum_{i=0}^{m-1}(-1)^i f_{n-i}\\
	&=\sum_{i=1}^{m-1}(-1)^{i-1} f_{n-i}+(-1)^{m-1} f_{n-m}+f_n+\sum_{i=1}^{m-1}(-1)^{i} f_{n-i}\\
	&=(-1)^{m-1} f_{n-m}+f_n.
	\end{align*}
For the proof of $(ii)$, it is clear that $\spn\{\{f_n+f_{n+1}\}_{n=1}^\infty \cup \{f_1\}\} \subseteq \spn\{f_n\}_{n=1}^\infty$.
 On the other hand, by $(i)$ (for $m=n-1$) we infer  $\spn\{f_n\}_{n=1}^\infty\subseteq \spn\{\{f_n+f_{n+1}\}_{n=1}^\infty \cup \{f_1\}\}$. This proves $(ii)$.
To prove $(iii)$, let  $\{c_n\}_{n=1}^k\subseteq \mathbb{C}$  such that $\sum_{n=1}^{k} c_n (f_n+f_{n+1})=0$. Then we have
		\begin{equation*}
		c_1f_1+\sum_{n=2}^{k}(c_{n-1}+c_n)f_n+c_k f_{k+1}=0.
		\end{equation*}
		Since $\{f_n\}_{n=1}^\infty$ is linearly independent, we get $c_1=c_k=0$ and $c_{n-1}+c_{n}=0$ for all $2\leqslant  n\leqslant  k$. Therefore, $c_n=0$ for all  $1\leqslant  n\leqslant  k$. This completes the proof of $(iii)$.
		\par
		To prove $(iv)$, let  $\{c_n\}_{n=1}^N\subseteq \mathbb{C}$  such that $\sum_{n=1}^N c_n f_n=0$. Then by $(i)$, we have
		\begin{align*}
		0=\sum_{n=1}^N c_n f_n
		&= c_1 f_1 +\sum_{n=2}^N c_nf_n\\
		&= c_1 f_1 +\sum_{n=2}^N c_n \Big( \sum_{i=0}^{n-2}(-1)^i (f_{n-i-1}+f_{n-i})+(-1)^{n-1} f_1\Big)\\
		&= \Big(c_1+\sum_{n=2}^N c_n(-1)^{n-1}\Big) f_1+\sum_{n=2}^N c_n \sum_{i=0}^{n-2} (-1)^i (f_{n-i-1}+f_{n-i})\\
		&= \Big(c_1+\sum_{n=2}^N c_n(-1)^{n-1}\Big) f_1+\sum_{i=0}^{N-2} \sum_{n=i+2}^{N} c_n (-1)^i (f_{n-i-1}+f_{n-i})\\
		&= \Big(c_1+\sum_{n=2}^N c_n(-1)^{n-1}\Big) f_1+\sum_{i=2}^{N} \sum_{n=i}^{N} c_n (-1)^{i-2} (f_{n-i+1}+f_{n-i+2})\\
		&= \Big(c_1+\sum_{n=2}^N c_n(-1)^{n-1}\Big) f_1+ \sum_{i=2}^N \sum_{n=0}^{N-i}c_{i+n} (-1)^{i}(f_{n+1}+f_{n+2})\\
		&= \Big(c_1+\sum_{n=2}^N c_n(-1)^{n-1}\Big) f_1+ \sum_{n=0}^{N-2} \Big(\sum_{i=2}^{N-n}c_{i+n} (-1)^{i}\Big)(f_{n+1}+f_{n+2}).
		\end{align*}
Since $\{f_1\} \cup\{f_n+f_{n+1}\}_{n=1}^\infty$ is linearly independent, we get
		\begin{align*}
		c_1+\sum_{k=2}^N c_k(-1)^{k-1}=0,\quad
		\sum_{i=2}^{N-n}c_{i+n} (-1)^{i}=0, \quad 0\leqslant  n\leqslant  N-2.
		\end{align*}
		Hence we conclude that $c_n=0$ for all $ n=1,2,..,N.$ Then $\{f_n\}_{n=1}^\infty$ is linearly independent.
\end{proof}
In the following, we give a sufficient condition for a sequence $F=\{f_n\}_{n=1}^\infty$ to have a Fibonacci representation.
\begin{thm}\label{mt}
	Let $F=\{f_n\}_{n=1}^\infty$ be a sequence in $\h$. If $\{f_n+f_{n+1}\}_{n=1}^\infty $ is linearly independent, then $F$ has a Fibonacci representation.
\end{thm}
\begin{proof}
	First we assume that $f_1\in \spn\{f_n+f_{n+1}\}_{n=1}^\infty$. Then by $(ii)$ of  Lemma \ref{lemf}, we have $\spn\{f_n+f_{n+1}\}_{n=1}^\infty=\spn\{f_n\}_{n=1}^\infty.$
	We define a linear operator $T:\spn\{f_n\}_{n=1}^\infty \rightarrow \spn\{f_n\}_{n=1}^\infty$ by
	\begin{equation}
	T(f_n+f_{n+1})=f_{n+2}; \quad n\geqslant  2.
	\end{equation}
	Since $\{f_n+f_{n+1}\}_{n=1}^\infty $ is linearly independent sequence, $T$ is well-defined and $F$ is represented by $T$.
	If $f_1\notin \spn\{f_n+f_{n+1}\}_{n=1}^\infty$, then $\{f_1\}\cup \{f_{n}+f_{n+1}\}_{n=1}^\infty $ is linearly independent and so by Lemma \ref{lemf} $(iv)$, $\{f_n\}_{n=1}^\infty$ is linearly independent. Hence we can define a linear operator $T:\spn\{f_n\}_{n=1}^\infty \rightarrow\spn\{f_n\}_{n=1}^\infty$  by
	\begin{equation*}
	Tf_n= \sum_{i=0}^{n} (-1)^{i} f_{n+1-i},\quad n\in\Bbb{N}.
	\end{equation*}
We show that $F$  is represented by $T$. Indeed,
	\begin{align*}
	Tf_n+Tf_{n+1}&=\sum_{i=0}^{n} (-1)^{i} f_{n+1-i}+\sum_{i=0}^{n+1} (-1)^{i} f_{n+2-i}\\
	&=\sum_{i=0}^{n} (-1)^{i} f_{n+1-i}+\sum_{i=1}^{n+1} (-1)^{i} f_{n+2-i} +f_{n+2}\\
	&=\sum_{i=0}^{n} (-1)^{i} f_{n+1-i}+\sum_{i=0}^{n} (-1)^{i+1} f_{n+1-i} +f_{n+2}=f_{n+2}.
	\end{align*}
\end{proof}
The following example shows that the converse of Theorem \ref{mt} is not satisfied in general.
\begin{example}
	The frame $F=\{f_n\}_{n=1}^\infty=\{e_1,e_2,e_3,e_2,e_2,e_4,e_5,e_6,...\}$  is represented by the linear operator  $T:\spn\{f_n\}_{n=1}^\infty \rightarrow\spn\{f_n\}_{n=1}^\infty$ given by
	\begin{align*}
	Te_1&=e_3-\dfrac{e_4}{2},\quad
	Te_2=\dfrac{e_4}{2},\quad
	Te_3=e_2-\dfrac{e_4}{2},\\
	Te_n&=\sum_{i=0}^{n-4} (-1)^i e_{n-i+1}+(-1)^{n-3}\dfrac{e_4}{2}, \quad n\geqslant  4.
	\end{align*}
	But $\{f_n+f_{n+1}\}_{n=1}^\infty=\{e_1+e_2,e_2+e_3,e_3+e_2,2e_2,...\}$ is not linearly independent.
\end{example}
\begin{cor}\label{lis}
	Let $F=\{f_n\}_{n=1}^\infty$ be a linear independent sequence in $\h$. Then $F$ has a Fibonachi representation.
\end{cor}
\begin{proof}
	It follows from Lemma \ref{lemf} $(iii)$ and Theorem \ref{mt}.
\end{proof}
Now, we provide  sufficient conditions to make the converse of Theorem \ref{mt}  satisfy.
\begin{thm}\label{mt2}
	Let $F=\{f_n\}_{n=1}^\infty$ be a complete sequence in an infinite dimensional Hilbert space $\h$ which is represented by $T$. If there exists $m\in \mathbb{N}$ such that $f_{m+1},Tf_1 \in \spn\{f_n\}_{n=1}^{m}$, then $\{f_n+f_{n+1}\}_{n=1}^\infty $ is linearly independent.
\end{thm}
\begin{proof}
	Suppose that $\{f_n+f_{n+1}\}_{n=1}^\infty $ is not linearly independent. Then there exists $n_0\in \mathbb{N}$ such that $f_{n_0} +f_{n_0+1}=\sum_{n=1}^{n_0-1} c_n(f_n+f_{n+1}).$ Hence
	\begin{equation}\label{14}
	f_{n_0+2}= T(f_{n_0} +f_{n_0+1})=\sum_{n=1}^{n_0-1} c_nf_{n+2} \in \spn\{f_n\}_{n=1}^{n_0+1}.\\
	\end{equation}
	Let $V=\spn\{f_n\}_{n=1}^l$, where $l=\max\{n_0+1 , m\}$.
	By \eqref{14} and $f_{m+1} \in \spn\{f_n\}_{n=1}^{m}$, we get $f_{l+1} \in V$. We show $V$ is invariant under $T$. Suppose that $f=\sum_{n=1}^l c_nf_n\in V$. By using $(i)$ of Lemma \ref{lemf}, we have
	\begin{align*}
	Tf=T\Big(\sum_{n=1}^l c_nf_n\Big)
	&= c_1 Tf_1 + \sum_{n=2}^l c_n Tf_n\\
	&=c_1Tf_1 + \sum_{n=2}^l c_n T\Big(\sum_{i=0}^{n}(-1)^i (f_{n-i-1}+f_{n-i})+(-1)^{n-1} f_1 \Big)\\
	&=c_1Tf_1 + \sum_{n=2}^l c_n \Big(\sum_{i=0}^{n}(-1)^i T(f_{n-i-1}+f_{n-i})+(-1)^{n-1} Tf_1\Big) \\
	&=\Big(c_1+\sum_{n=2}^l c_n(-1)^{n-1}\Big)Tf_1 + \sum_{n=2}^l c_n\sum_{i=0}^{n}(-1)^i f_{n-i+1}.
	\end{align*}
	Since $Tf_1 \in \spn\{f_n\}_{n=1}^{m}\subseteq V$ and $f_{l+1}\in V$, the above argument proves that $V$ is invariant under $T$. Therefore    $f_n \in V $ for all $n\geqslant  l+1$ and consequently $\spn\{f_n\}_{n=1}^\infty=V$. Since $\{f_n\}_{n=1}^\infty$ is complete in $\h$, we have $\h =\overline{\spn}\{f_n\}_{n=1}^\infty=\overline{V}=V$ which is in contradiction to $\dim \h=\infty$.
\end{proof}
\begin{prop}\label{remcom}
	 Let $\{f_n\}_{n=1}^{\infty}$ be  a complete and linearly dependent sequence with $f_1\neq0$  in an infinite dimensional Hilbert space $\h$. Then there exists $m\geqslant2$ such that $f_m\in \spn\{f_n\}_{n=1}^{m-1}$ and  $f_{m+1}\notin \spn\{f_n\}_{n=1}^{m}$.
\end{prop}
\begin{proof}
  Since $\{f_n\}_{n=1}^{\infty}$ is  linearly dependent, there exists $k\geqslant2$ such that $f_k\in\spn\{f_n\}_{n=1}^{k-1}$. We claim that there exists an integer $l>k$ such that $f_l\notin\spn\{f_n\}_{n=1}^{l-1}$. If $f_l\in\spn\{f_n\}_{n=1}^{l-1}$ for each $l>k$, then $f_{k+1}\in\spn\{f_n\}_{n=1}^{k-1}$ because $f_k\in\spn\{f_n\}_{n=1}^{k-1}$ and $f_{k+1}\in\spn\{f_n\}_{n=1}^{k}$. Hence by induction we get $f_{l}\in\spn\{f_n\}_{n=1}^{k-1}$ for each $l>k$. Therefore $\spn\{f_n\}_{n=1}^{\infty}=\spn\{f_n\}_{n=1}^{k-1}$. Since $\{f_n\}_{n=1}^{\infty}$ is complete and $\dim \h=\infty$, the contradiction is achieved.
  Now, let $i\in\Bbb{N}$ be the smallest number such that $f_{k+i}\notin\spn\{f_n\}_{n=1}^{k+i-1}$. Putting $m=k+i-1$, we get $f_m\in \spn\{f_n\}_{n=1}^{m-1}$ and  $f_{m+1}\notin \spn\{f_n\}_{n=1}^{m}$.
\end{proof}
\begin{prop}\label{in}
	Let $F=\{f_n\}_{n=1}^\infty $ be a sequence in $\h$ which is represented by $T$. Suppose that $f_m\in \spn\{f_n\}_{n=1}^{m-1}$ and $f_{m+1} \notin \spn\{f_n\}_{n=1}^{m-1}$ for some integer $m\geqslant 2$. Then $Tf_i\in \spn\{f_n\}_{n=3}^{m+1}$ for $1\leqslant  i \leqslant  m$.
\end{prop}
\begin{proof}
	By the assumption, we have $f_m=\sum_{n=1}^{m-1}c_n f_n$, so
	\begin{align*}
	\sum_{n=1}^{m-2}&\Big(\sum_{i=0}^{n-1}(-1)^i c_{n-i}\Big)(f_n+f_{n+1})\\
	&=\sum_{n=1}^{m-2}\Big(\sum_{i=0}^{n-1}(-1)^i c_{n-i}\Big)f_n + \sum_{n=1}^{m-2}\Big(\sum_{i=0}^{n-1}(-1)^i c_{n-i}\Big)f_{n+1}\\
	&=\sum_{n=1}^{m-2}\Big(\sum_{i=0}^{n-1}(-1)^i c_{n-i}\Big)f_n + \sum_{n=2}^{m-1}\Big(\sum_{i=0}^{n-2}(-1)^i c_{n-i-1}\Big)f_{n}\\
	&=c_1f_1+\sum_{n=2}^{m-2}\Big(\sum_{i=0}^{n-1}(-1)^i c_{n-i}\Big)f_n\\
	&\quad +\sum_{n=2}^{m-2}\Big(\sum_{i=0}^{n-2}(-1)^i c_{n-i-1}\Big)f_{n}+\Big(\sum_{i=0}^{m-3}(-1)^i c_{m-i-2}\Big)f_{m-1}\\
	&=c_1f_1+\sum_{n=2}^{m-2} \Big(\sum_{i=0}^{n-1}(-1)^i c_{n-i} + \sum_{i=0}^{n-2}(-1)^i c_{n-i-1}\Big) f_{n}\\
	&\quad +\Big(\sum_{i=0}^{m-3}(-1)^i c_{m-i-2}\Big)f_{m-1}\\
	&=c_1f_1+\sum_{n=2}^{m-2} \Big(c_n+\sum_{i=1}^{n-1}(-1)^i c_{n-i}+\sum_{i=1}^{n-1}(-1)^{i-1} c_{n-i}\Big)f_{n}\\
	&\quad+\Big(\sum_{i=0}^{m-3}(-1)^i c_{m-i-2}\Big)f_{m-1}\\
	&=c_1f_1+\sum_{n=2}^{m-2} c_nf_{n}+\Big(\sum_{i=0}^{m-3}(-1)^i c_{m-i-2}\Big)f_{m-1}\\
	&=f_{m}+\Big(-c_{m-1}+\sum_{i=0}^{m-3}(-1)^i c_{m-i-2}\Big)f_{m-1}\\
	&=f_{m}+\Big(\sum_{i=0}^{m-2}(-1)^{i-1} c_{m-i-1}\Big)f_{m-1},
	\end{align*}
	thus
	\begin{align}\label{554}
	f_{m-1}+f_{m}=\sum_{n=1}^{m-2}\Big(\sum_{i=0}^{n-1}(-1)^i c_{n-i}\Big)(f_n+f_{n+1})
	+\Big(1-\sum_{i=0}^{m-2}(-1)^{i-1} c_{m-i-1}\Big)f_{m-1}.
	\end{align}
	Since $F$ is represented by $T$, the equality \eqref{554} implies that
	\begin{equation}\label{12}
	f_{m+1}=\sum_{n=1}^{m-2}\Big(\sum_{i=0}^{n-1}(-1)^i c_{n-i}\Big)f_{n+2}+\Big(1-\sum_{i=0}^{m-2}(-1)^{i-1} c_{m-i-1}\Big)Tf_{m-1}.
	\end{equation}
	If $1-\sum_{i=0}^{m-2}(-1)^{i-1} c_{m-i-1}=0$, then $f_{m+1}\in \spn\{f_n\}_{n=3}^{m}\subseteq\spn\{f_n\}_{n=1}^{m-1}$ which is a contradiction.
	Hence \eqref{12} implies that
	\begin{equation}\label{13}
	Tf_{m-1}=\dfrac{f_{m+1}-\sum_{n=1}^{m-2}\Big(\sum_{i=0}^{n-1}(-1)^i c_{n-i}\Big)f_{n+2}}{1-\sum_{i=0}^{m-2}(-1)^{i-1} c_{m-i-1}}\in\spn\{f_n\}_{n=3}^{m+1}.
	\end{equation}
	Also, by $(i)$ of Lemma \ref{lemf}, for $1\leqslant  j \leqslant  m-1$, we have
	\begin{align*}\label{15}
	f_{m+1}=Tf_{m}+ Tf_{m-1}&=T\Big(\sum_{i=0}^{j-1}(-1)^i (f_{m-i-1}+f_{m-i})+(-1)^{j} f_{m-j}\Big)+Tf_{m-1}\\
	&=\sum_{i=0}^{j-1}(-1)^i f_{m-i+1}+(-1)^{j} Tf_{m-j}+Tf_{m-1}.
	\end{align*}
	Therefore
	\begin{equation}\label{116}
	Tf_{m-j}=(-1)^{j} \big(f_{m+1}-Tf_{m-1}-\sum_{i=0}^{j-1}(-1)^i f_{m-i+1}\big).
	\end{equation}
	Hence it follows from  \eqref{13} and \eqref{116} that   $Tf_i \in \spn\{f_n\}_{n=3}^{m+1}$ for each $ 1\leqslant  i \leqslant  m-1$.
\end{proof}
\begin{cor}\label{in2}
	Let $F=\{f_n\}_{n=1}^\infty $ be a sequence in $\h$ which is represented by $T$. Suppose that $f_m\in \spn\{f_n\}_{n=1}^{m-1}$ and $f_{m+1} \notin \spn\{f_n\}_{n=1}^{m-1}$ for some $m\in \mathbb{N}$. Then, $Tf_{m+i}\in \spn\{f_n\}_{n=3}^{m+i+1}$ for each $i \in \mathbb{N}$.
\end{cor}
\begin{proof}
	Since $Tf_{m+i}=f_{m+i+1}-Tf_{m+i-1}$, the result follows by induction on $i$ and Proposition \ref{in}.
\end{proof}
\begin{cor}\label{ranspn}
	Let $F=\{f_n\}_{n=1}^\infty $ be a complete sequence in an infinite dimensional Hilbert space $\h$.
	\begin{enumerate}
		\item [$(i)$] If $F$ is linearly independent, then it has a Fibonacci representation $T$ such that $\ran(T)=\spn\{f_n\}_{n=3}^\infty.$
		\item [$(ii)$] If $F$ is linearly dependent, then for any Fibonacci representation $T$ of $F$ we have $\ran(T)=\spn\{f_n\}_{n=3}^\infty.$
	\end{enumerate}
\end{cor}
\begin{proof}
	First we note that if $F$ is represented by $T$, then $f_n=T(f_{n-1}+f_{n-2})\in \ran (T)$ for every $n\geqslant  3$, and consequently   $\spn\{f_n\}_{n=3}^\infty\subseteq \ran( T)$.
\par
	To prove $(i)$, consider the linear operator $T:\spn\{f_n\}_{n=1}^\infty \rightarrow \spn\{f_n\}_{n=1}^\infty$ defined by
	\begin{align*}
	Tf_1=Tf_2=\frac{1}{2}f_3,\quad
	Tf_n= \sum_{i=0}^{n-3} (-1)^{i} f_{n+1-i}+ \frac{(-1)^n}{2}f_3, \quad n\geqslant  3.
	\end{align*}
	Then
	\begin{align*}
	&Tf_1+Tf_2=f_3,\quad Tf_2+Tf_3=f_4,\\ &Tf_{n}+Tf_{n+1}=\sum_{i=0}^{n-3}(-1)^i f_{n-i+1}+\sum_{i=0}^{n-2}(-1)^i f_{n-i+2}=f_{n+2}, \quad n\geqslant  3.
	\end{align*}
	Hence $F$ is represented by $T$ and it is obvious that $\ran (T)\subseteq \spn\{f_n\}_{n=3}^\infty$.
	In order to prove $(ii)$, by Proposition \ref{remcom} there exists $m\geqslant 2$ such that $f_m\in \spn\{f_n\}_{n=1}^{m-1}$ and  $f_{m+1}\notin \spn\{f_n\}_{n=1}^{m-1}$. If $F$ is represented by $T$, then  by Proposition \ref{in} and Corollary \ref{in2}  we have $\ran(T)\subseteq \spn\{f_n\}_{n=3}^\infty.$
\end{proof}
In Theorem \ref{mt2}, we showed that $\{f_{n}+f_{n+1}\}_{n=1}^\infty  $ is linearly independent under some conditions. In the following, we show that (under some conditions) by removing finitely many elements of $\{f_{n}+f_{n+1}\}_{n=1}^\infty  $ the remaining elements will be linearly independent.
\begin{thm}\label{mt3}
	Let $F=\{f_n\}_{n=1}^\infty$ be a complete sequence in an infinite dimensional Hilbert space $\h$ which is represented by $T$. Then there exists $m\in \mathbb{N} $ such that $\{f_{m+n}+f_{m+n+1}\}_{n=1}^\infty  $ is linearly independent.
\end{thm}
\begin{proof}
	If $F$ is linearly independent, then the result follows by $(iii)$ of Lemma \ref{lemf}.
	Suppose that $F$ is linearly dependent. Then  by Proposition \ref{remcom} and Proposition \ref{in}, there exists $m\geqslant 2$ such that  $f_m\in \spn\{f_n\}_{n=1}^{m-1}$, $f_{m+1}\notin \spn\{f_n\}_{n=1}^{m-1}$ and $Tf_1 \in \spn\{f_n\}_{n=3}^{m+1}$. We prove $\{f_{m+n}+f_{m+n+1}\}_{n=1}^\infty  $ is  linearly independent. Suppose by contradiction that $\{f_{m+n}+f_{m+n+1}\}_{n=1}^\infty  $ is not linearly independent. Then there exists $j\in \mathbb{N}$ such that $f_{m+j} +f_{m+j+1}=\sum_{n=1}^{j-1} c_n(f_{m+n}+f_{m+n+1}).$ Hence we have
	\begin{equation}\label{17}
	f_{m+j+2}= T(f_{m+j} +f_{m+j+1})=\sum_{n=1}^{j-1} c_nf_{m+n+2} \in \spn\{f_n\}_{n=1}^{m+j+1}.\\
	\end{equation}
	Let $V=\spn\{f_n\}_{n=1}^{m+j+1}$.  We show that $V$ is invariant under $T$. Let $f=\sum_{n=1}^{m+j+1} c_nf_n\in V$. Then by $(i)$ of Lemma \ref{lemf}, we have
	\begin{align*}
	Tf&= c_1 Tf_1 + \sum_{n=2}^{m+j+1} c_n Tf_n\\
	&=c_1Tf_1 + \sum_{n=2}^{m+j+1} c_nT\Big(\sum_{i=0}^{n-2}(-1)^i (f_{n-i-1}+f_{n-i})+(-1)^{n-1} f_1 \Big)\\
	&=c_1Tf_1 + \sum_{n=2}^{m+j+1} c_n \Big(\sum_{i=0}^{n-2}(-1)^i f_{n-i+1}+(-1)^{n-1} Tf_1\Big) \\
	&=\Big(c_1+\sum_{n=2}^{m+j+1} c_n(-1)^{n-1}\Big)Tf_1 + \sum_{n=2}^{j+m+1} c_n\sum_{i=0}^{n-2}(-1)^i f_{n-i+1}.
	\end{align*}
Using $Tf_1 \in \spn\{f_n\}_{n=3}^{m+1} \subseteq V$ and  \eqref{17}, we get $Tf \in V$.
	Then we conclude   $f_n \in V $ for all $n\geqslant  m+j+2$. Thus, $\spn\{f_n\}_{n=1}^\infty=V$ and since $\{f_n\}_{n=1}^\infty$ is complete in $\h$, we have $\h =\overline{\spn}\{f_n\}_{n=1}^\infty=\overline{V}=V$ which is a contradiction.
\end{proof}
\section{Fibonacci Representation Operators}
In a frame that indeed has  the form $\{T^n\varphi\}_{n=0}^\infty$, where
$T \in B(\h)$  and $\varphi\in\h$, all sequence members are represented by iterative actions of $T$ on $\varphi$.
In the case where $\{f_n\}_{n=1}^\infty$ has a Fibonacci representation operator $T$, we expect (Theorem \ref{fib}) all members of the sequence $\{f_n\}_{n=1}^\infty$ to be identified in terms of iterative actions of $T$ on elements $f_1$ and $f_2$.
In this section, we present some results concerning Fibonacci representation operators. One of the  results characterizes  types of frame which can be represented in terms of a bounded operator $T$.
\begin{notation}
	 $[x]$ denotes the  integer part of $x\in\Bbb{R}$ and $\vv{n}{k}:=\frac{n!}{k!(n-k)!}$ for integers $0\leqslant  k\leqslant  n$. We let $\vv{n}{k}:=0$ when $k>n$ or $k<0$.
\end{notation}
\begin{thm}\label{fib}
	Let $T:\spn\{f_n\}_{n=1}^\infty \rightarrow \spn\{f_n\}_{n=1}^\infty $ be a linear operator, then the following statements are equivalent:
	\begin{enumerate}
		\item [$(i)$] $F=\{f_n\}_{n=1}^\infty$ is represented by $T$.
		\item [$(ii)$] $Tf_1+Tf_2=f_3$ and
		\begin{equation}\label{rep}
		f_n=\sum_{i=a_n}^{2a_n}\Big( \vv{i+b_n}{2i-2a_n+b_n} T^{i+b_n}f_2+ \vv{i+b_n}{2i-2a_n+b_n+1} T^{i+b_n+1}f_1\Big),\quad n\geqslant  4,
		\end{equation}
		where  $a_n=[\frac{n-1}{2}]$ , $b_n=n-2a_n-2$.
	\end{enumerate}
\end{thm}
\begin{proof}$(i)\Rightarrow (ii)$
	We prove \eqref{rep} by induction on $n$. For $n=4$, we have $a_4=1$ and $b_4=0 $. Then
	\begin{align*}
	&\vv{1}{0} Tf_2+ \vv{1}{1} T^{2}f_1+ \vv{2}{2} T^{2}f_2+ \vv{2}{3} T^{3}f_1\\
	&=Tf_2+T(Tf_1+Tf_2)=Tf_2+Tf_3=f_4.
	\end{align*}
Now, assume that $k>4$ and \eqref{rep} holds for all $n\leqslant  k$ and we prove \eqref{rep} for $n=k+1$.
	If $k+1$ is even, then $b_k=-1,~ b_{k-1}=b_{k+1}=0$ and $a_{k+1}=a_k=1+a_{k-1}$. Hence
	\begin{align*}
	f_{k+1}&=Tf_k+Tf_{k-1}\\
	&=\sum_{i=a_k}^{2a_k} \Big( \vv{i-1}{2i-2a_k-1} T^{i}f_2+ \vv{i-1} {2i-2a_k} T^{i+1}f_1 \Big)\\
	&\quad+\sum_{i=a_{k-1}}^{2a_{k-1}} \Big(\vv{i}{2i-2a_{k-1}} T^{i+1}f_2+ \vv{i}{2i-2a_{k-1}+1} T^{i+2}f_1\Big)\\
	&=\sum_{i=a_k}^{2a_{k}} \Big(\vv{i-1}{2i-2a_{k}-1}T^{i}f_2+ \vv{i-1}{2i-2a_{k}} T^{i+1}f_1 \Big)\\
	&\quad+\sum_{i=a_k}^{2a_k-1} \Big(\vv{i-1}{2i-2a_{k}} T^{i}f_2+ \vv{i-1}{2i-2a_{k}+1} T^{i+1}f_1\Big)\\
	&=\sum_{i=a_k}^{2a_{k}} \Big(\vv{i-1}{2i-2a_{k}-1}T^{i}f_2+ \vv{i-1}{2i-2a_{k}} T^{i+1}f_1 \Big)\\
	&\quad+\sum_{i=a_k}^{2a_k} \Big(\vv{i-1}{2i-2a_{k}} T^{i}f_2+ \vv{i-1}{2i-2a_{k}+1} T^{i+1}f_1\Big)\\
	&=\sum_{i=a_k}^{2a_k} \Big(\vv{i}{2i-2a_{k}}  T^{i}f_2+ \vv{i}{2i-2a_{k}+1} T^{i+1}f_1\Big).
	\end{align*}
Since $b_{k+1}=0$ and $a_{k+1}=a_k$, we obtain \eqref{rep}.
	If $k+1$ is odd, then $b_k=0,~ b_{k-1}=b_{k+1}=-1$ and $1+a_{k-1}=1+a_k=a_{k+1}$. Hence
	\begin{align*}
	f_{k+1}&=Tf_k+Tf_{k-1}\\
	&=\sum_{i=a_{k}}^{2a_{k}} \Big( \vv{i}{2i-2a_{k}} T^{i+1}f_2+ \vv{i}{2i-2a_{k}+1} T^{i+2}f_1\Big)\\
	&\quad+\sum_{i=a_{k-1}}^{2a_{k-1}} \Big( \vv{i-1}{2i-2a_{k-1}-1}T^{i}f_2+ \vv{i-1}{2i-2a_{k-1}} T^{i+1}f_1\Big)\\
	&=\sum_{i=1+a_{k+1}}^{2a_{k+1}} \Big( \vv{i-2}{2i-2a_{k+1}-2} T^{i-1}f_2+ \vv{i-2}{2i-2a_{k+1}-1} T^{i}f_1\Big)\\
	&\quad+\sum_{i=a_{k+1}}^{2a_{k+1}-1} \Big( \vv{i-2}{2i-2a_{k+1}-1} T^{i-1}f_2+ \vv{i-2}{2i-2a_{k+1}} T^{i}f_1\Big)\\
	&=\sum_{i=a_{k+1}}^{2a_{k+1}} \Big( \vv{i-2}{2i-2a_{k+1}-2} T^{i-1}f_2+ \vv{i-2}{2i-2a_{k+1}-1} T^{i}f_1\Big)\\
	&\quad+\sum_{i=a_{k+1}}^{2a_{k+1}} \Big( \vv{i-2}{2i-2a_{k+1}-1} T^{i-1}f_2+ \vv{i-2}{2i-2a_{k+1}} T^{i}f_1\Big)\\
	&=\sum_{i=a_{k+1}}^{2a_{k+1}} \Big( \vv{i-1}{2i-2a_{k+1}-1} T^{i-1}f_2+ \vv{i-1}{2i-2a_{k+1}} T^{i}f_1\Big).
	\end{align*}
Hence we get \eqref{rep}. To prove $(ii)\Rightarrow (i)$, there are two possibilities.
	If $n>4$ is odd, then  $b_n=-1,~ b_{n-1}=b_{n+1}=0$ and $a_{n+1}=a_n=1+a_{n-1}$. Hence
	\begin{align*}
	Tf_n&=\sum_{i=a_n}^{2a_n}\Big( \vv{i-1}{2i-2a_{n}-1} T^{i}f_2+\vv{i-1}{2i-2a_{n}}  T^{i+1}f_1\Big),\\
Tf_{n-1}&=\sum_{i=a_{n-1}}^{2a_{n-1}}\Big( \vv{i}{2i-2a_{n-1}} T^{i+1}f_2+ \vv{i}{2i-2a_{n-1}+1}  T^{i+2}f_1\Big)\\
	&=\sum_{i=a_n}^{2a_n-1} \Big(\vv{i-1}{2i-2a_{n}} T^{i}f_2+ \vv{i-1}{2i-2a_{n}+1} T^{i+1}f_1\Big)\\
&=\sum_{i=a_n}^{2a_n} \Big(\vv{i-1}{2i-2a_{n}} T^{i}f_2+ \vv{i-1}{2i-2a_{n}+1} T^{i+1}f_1\Big).
	\end{align*}
	Therefore,
	\begin{align*}
	Tf_n+Tf_{n-1}=\sum_{i=a_n}^{2a_n} \Big(\vv{i}{2i-2a_{n}} T^{i}f_2+ \vv{i}{2i-2a_{n}+1}  T^{i+1}f_1\Big)
	=f_{n+1}.
	\end{align*}
If $n\geqslant 4$ is even, the argument is similar to the previous case.
\end{proof}
\begin{rem}
	We recall that
\[\ell^2(\h):=\Big\{\{f_n\}_{n=1}^\infty\subseteq \h: \sum_{n=1}^\infty \|f_n\|^2<\infty\Big\},\]
 and $\mathcal{T}_L, \mathcal{T}_R:\ell^2(\h) \rightarrow \ell^2(\h)$ are bounded linear operators defined by
 \[\mathcal{T}_L \{f_n\}_{n=1}^\infty= \{f_{n+1}\}_{n=1}^\infty ,\quad \mathcal{T}_R \{f_n\}_{n=1}^\infty= \{0,f_1,f_2,f_3,...\}.\]
\end{rem}
\begin{prop}\label{bound}
	Let $F=\{f_n\}_{n=1}^\infty $ be a Bessel sequence in $\h$ which is represented by $T_0$ and let $M=\{f_n+f_{n+1}\}_{n=1}^\infty$ be a frame for $\h$.  Then $\ker T_M\subseteq \ker T_{\mathcal{T}_L^2F}$ if and only if $T:=T_0|_{\spn\{f_n+f_{n+1}\}_{n=1}^\infty}$ is bounded with $\|T\|\leqslant  \sqrt{\frac{B_F}{A_M}}$, where $B_F$ is a Bessel bound for $F$ and $A_M$ is a lower frame bound for $M$.
\end{prop}
\begin{proof}
	Let $\ker T_M\subseteq \ker T_{\mathcal{T}_L^2F}$ and
	$f=\sum_{n=1}^{k} c_n (f_n+f_{n+1})$, where $\{c_n\}_{n=1}^\infty  \in\ell^2$ with $c_n=0$ for $n\geqslant  k+1$. Then
	\begin{align*}
	Tf=T\Big(\sum_{n=1}^{k} c_n (f_n+f_{n+1})\Big)
	&=\sum_{n=1}^{k} c_n T(f_n+f_{n+1})\\
	&=\sum_{n=1}^{\infty} c_n f_{n+2}\\
&=\sum_{n=1}^{\infty} d_n f_{n+2} +\sum_{n=1}^{\infty} r_n f_{n+2},
	\end{align*}
	where $\{d_n\}_{n=1}^\infty  \in \ker T_M \subseteq \ker T_{\mathcal{T}_L^2F}$ and $\{r_n\}_{n=1}^\infty  \in (\ker T_M)^\perp$.
	Since $\{d_n\}_{n=1}^\infty  \in \ker T_{\mathcal{T}_L^2F}$, we have $\sum_{n=1}^{\infty} d_n f_{n+2}=0$ and consequently $Tf=\sum_{n=1}^{\infty} r_n f_{n+2}$. Therefore
	by applying \cite[Lemma 5.5.5]{Ole_Book}, we have
	\begin{align*}
	\|Tf\|^2&=\Big\|\sum_{n=1}^{\infty} r_n f_{n+2}\Big\|^2 \leqslant   B_F\sum_{n=1}^{\infty} |r_n|^2 \leqslant   \frac{B_F}{A_M}\Big\|\sum_{n=1}^{\infty} r_n (f_{n}+f_{n+1})\Big\|^2\\
	&= \frac{B_F}{A_M}\Big\|\sum_{n=1}^{\infty} r_n (f_{n}+f_{n+1}) + \sum_{n=1}^{\infty} d_n (f_{n}+f_{n+1})\Big\|^2\\
	&=\frac{B_F}{A_M}\Big\|\sum_{n=1}^{k} c_n (f_{n}+f_{n+1})\Big\|^2=\frac{B_F}{A_M}\|f\|^2.
	\end{align*}
	For the other implication, let $ \{c_n\}_{n=1}^\infty \in \ker T_{M}$. Since $\sum_{n=1}^\infty c_n(f_n+f_{n+1})=0$ and $T$ is bounded, we have
\[0=T\Big(\sum_{n=1}^\infty c_n(f_n+f_{n+1})\Big)=\sum_{n=1}^\infty c_n f_{n+2}.\]
 This means $\{c_n\}_{n=1}^\infty \in \ker T_{\mathcal{T}_L^2F}$.
\end{proof}
\begin{rem}
	In Theorem \ref{exbnd}, the invariance of  $\ker T_F$ under the right-shift operator $\mathcal{T}$ is sufficient condition for the boundedness of $T$. It is obvious that the invariance of  $\ker T_F$ under $\mathcal{T}$ is equivalent with
$\ker T_F\subseteq \ker T_{\mathcal{T}_L F}$. In fact, for $\{c_n\}_{n=1}^\infty\in \ell^2$ we have
	$\mathcal{T}(\{c_n\}_{n=1}^\infty) \in \ker T_F$ if and only if $\{c_n\}_{n=1}^\infty\in \ker T_{\mathcal{T}_LF}.$
\end{rem}
\begin{prop}
	Let $F=\{f_n\}_{n=1}^\infty $ be a Bessel sequence in $\h$ which is represented by $T\in B(\h)$ and $M=\{f_n+f_{n+1}\}_{n=1}^\infty$ be a frame for $\h$. Then $T$ is injective if and only if  $ \ker T_{\mathcal{T}_L^2F}\subseteq \ker T_M$
\end{prop}
\begin{proof}
Let $T$ be injective and $\{c_n\}_{n=1}^\infty\in \ker T_{\mathcal{T}_L^2F}$. Then
\[T\Big(\sum_{n=1}^\infty c_n(f_n+f_{n+1})\Big)= \sum_{n=1}^\infty c_nf_{n+2}=0.\]
Since $T$ is injective, we get $\sum_{n=1}^\infty c_n(f_n+f_{n+1})=0$ and consequently $\{c_n\}_{n=1}^\infty\in \ker T_M$. Conversely, assume that $f\in\h$ and $Tf=0$. Since $M=\{f_n+f_{n+1}\}_{n=1}^\infty$ is a frame for $\h$, we  have $f=\sum_{n=1}^\infty c_n(f_n+f_{n+1})$ for some $\{c_n\}_{n=1}^\infty\in\ell^2$. Then
	\[ \sum_{n=1}^\infty c_nf_{n+2}=T\Big(\sum_{n=1}^\infty c_n(f_n+f_{n+1})\Big)=Tf=0.\]
Hence $\{c_n\}_{n=1}^\infty\in \ker T_{\mathcal{T}_L^2F}$ and consequently $\{c_n\}_{n=1}^\infty\in \ker T_M$. This means $f=\sum_{n=1}^\infty c_n(f_n+f_{n+1})=0$ and the proof is completed.
\end{proof}
\begin{prop}
	Let $\{f_n\}_{n=1}^\infty $ be represented by $T$. Then the following hold:
\begin{itemize}
  \item [$(i)$] If $K\in B(\h)$ is injective and has closed range, then $\{Kf_n\}_{n=1}^\infty$ has a Fibonacci representation.
  \item [$(ii)$] If $K\in B(\h)$ is surjective, then $\{K^*f_n\}_{n=1}^\infty$ and $\{KK^*f_n\}_{n=1}^\infty$ have Fibonacci representations.
\end{itemize}
\end{prop}
\begin{proof}
	$(i)$ By Open Mapping Theorem, there exists a bonded linear operator $S:\ran(K)\rightarrow \h$ such that $SK=I_{\h}$. Therefore
	\begin{equation*}
	KTS(Kf_n+Kf_{n-1})=KT(f_n+f_{n-1})=Kf_{n+1},\quad n\geqslant 2.
\end{equation*}
\par
To prove $(ii)$, by \cite[Lemma 2.4.1]{Ole_Book}, $K^*$ is injective and has closed range. Also $KK^*$ is invertible.  Then $(i)$ implies $(ii)$.
\end{proof}
\begin{prop}
	Let $\{f_n\}_{n=1}^\infty $ be a frame for $\h$ and represented by $T\in B(\h)$. If $Tf_1\in\spn\{f_n\}_{n=3}^\infty$,  then $\ran(T)$ is closed and $\ran(T)=\overline{\spn}\{Tf_n\}_{n=1}^\infty=\overline{\spn}\{f_n\}_{n=3}^\infty$.
\end{prop}
\begin{proof}
	For each $f\in\h$, there exists $\{c_n\}_{n=1}^\infty\in\ell^2$ such that $f=\sum_{n=1}^\infty c_nf_n$. Then
$Tf=\sum_{n=1}^\infty c_nTf_n\in \overline{\spn}\{Tf_n\}_{n=1}^\infty$, and therefore $\ran(T)\subseteq\overline{\spn}\{Tf_n\}_{n=1}^\infty$. On the other hand, for $g\in\overline{\spn}\{f_n\}_{n=3}^\infty$ there exists $\{c_n\}_{n=1}^\infty\in\ell^2$ such that
\[g=\sum_{n=1}^\infty c_nf_{n+2}=\sum_{n=1}^\infty c_n T(f_n+f_{n+1})=T\Big(\sum_{n=1}^\infty c_n (f_n+f_{n+1})\Big)\in\ran (T).\]
Then $\overline{\spn}\{f_n\}_{n=3}^\infty\subseteq\ran(T)$. Since $Tf_1\in\spn\{f_n\}_{n=3}^\infty$, we can now apply $(ii)$ of Lemma \ref{lemf}
to conclude that
	\begin{equation*}
	\spn\{Tf_n\}_{n=1}^\infty=\spn\Big\{\{Tf_1\}\cup\{Tf_n+Tf_{n+1}\}_{n=1}^\infty\Big\}=\spn\{f_n\}_{n=3}^\infty.
\end{equation*}
Therefore $\ran(T)=\overline{\spn}\{Tf_n\}_{n=1}^\infty=\overline{\spn}\{f_n\}_{n=3}^\infty$.
\end{proof}
\begin{prop}
	Let $\{f_n\}_{n=1}^\infty $ be a linearly dependent frame sequence  represented by $T\in B(\K)$, where $\K=\overline{\spn}\{f_n\}_{n=1}^\infty$ is an infinite dimensional Hilbert space. Then $\ran(T)$ is closed and $\ran(T)=\overline{\spn}\{f_n\}_{n=3}^\infty$.
\end{prop}
\begin{proof}
Let $T_0$  be the restriction of $T$ on  $\spn\{f_n\}_{n=1}^\infty$. Then by $(ii)$ of Corollary \ref{ranspn}, we have
$\ran (T_0)=\spn\{f_n\}_{n=3}^\infty$, and therefore $\ran(T)\subseteq\overline{\spn}\{f_n\}_{n=3}^\infty$.
On the other hand, Since $\{f_n\}_{n=1}^\infty $ is a frame sequence, for each $f\in\overline{\spn}\{f_n\}_{n=3}^\infty$, there exists $\{c_n\}_{n=1}^\infty\in\ell^2$ such that
\[f=\sum_{n=1}^\infty c_nf_{n+2}=\sum_{n=1}^\infty c_n(Tf_{n}+Tf_{n+1})=T\Big(\sum_{n=1}^\infty c_n(f_{n}+f_{n+1})\Big)\in\ran (T).\]
Hence $\overline{\spn}\{f_n\}_{n=3}^\infty\subseteq \ran(T)$ and this completes the proof.
\end{proof}
\begin{thm}\label{T=S}
	Let $\{f_n\}_{n=1}^\infty $ be represented by $T$ and $S$. If $Tf_1=Sf_1$, then $T=S$.
\end{thm}
\begin{proof}
	Since $Tf_1=Sf_1$ and $T(f_n+f_{n+1})=f_{n+2}=S(f_n+f_{n+1})$ for all $n\in\Bbb{N}$, we get $Tf_n=Sf_n$ for all $n\in\Bbb{N}$ (we can use $(i)$ of  Lemma \ref{lemf}). This proves $T=S$ on $\spn \{f_n\}_{n=1}^\infty$.
\end{proof}
\begin{cor}
Let $\{f_n\}_{n=1}^\infty $ be represented by $T$ and $S$. If $f_1\in\spn \{f_n+f_{n+1}\}_{n=k}^\infty$ for some $k\in\Bbb{N}$, then $T=S$.
\end{cor}
\begin{proof}
	Since $f_1\in\spn \{f_n+f_{n+1}\}_{n=k}^\infty$, we have $f_1=\sum_{n=k}^m c_n(f_n+f_{n+1})$ for some scalars $c_k,\cdots,c_m$. Then
\[Tf_1=T\Big(\sum_{n=k}^m c_n(f_n+f_{n+1})\Big)=\sum_{n=k}^m c_nT(f_n+f_{n+1})=\sum_{n=k}^m c_nf_{n+2}=\sum_{n=k}^m c_nS(f_n+f_{n+1})=Sf_1.\]
 Therefore $T=S$ by Theorem \ref{T=S}.
\end{proof}

\end{document}